\documentclass[reqno, english,11pt]{smfart}

\usepackage{amsfonts, amsthm, amsmath, amssymb}

\usepackage{a4wide}

\RequirePackage{mathrsfs} \let\mathcal\mathscr

\numberwithin{equation}{section}

\renewcommand{\phi}{\varphi}
\newcommand{\phid}{\varphi^*}

\newcommand{\PP}{\mathbb{P}}
\renewcommand{\AA}{\mathbb{A}}

\newcommand{\ZZ}{\mathbb{Z}}

\newcommand{\NN}{\mathbb{N}}
\newcommand{\QQ}{\mathbb{Q}}
\newcommand{\RR}{\mathbb{R}}
\newcommand{\CC}{\mathbb{C}}

\newcommand{\cA}{\mathcal{A}}
\newcommand{\cS}{\mathcal{S}}

\newtheorem{thm}{Theorem}
\newtheorem*{thm*}{Theorem}
\newtheorem{lemma}{Lemma}
\newtheorem*{cor*}{Corollary}
\newtheorem{prop}{Proposition}

\renewcommand{\leq}{\leqslant}
\renewcommand{\le}{\leqslant}
\renewcommand{\geq}{\geqslant}
\renewcommand{\ge}{\geqslant}

\newcommand{\ma}{\mathbf}

\newcommand{\x}{\mathbf{x}}

\newcommand{\al}{\alpha}

\newcommand{\e}{\ensuremath{\mathrm e}}
\newcommand{\bet}{\boldsymbol{\eta}}

\newcommand{\bal}{\boldsymbol{\alpha}}
\newcommand{\ve}{\varepsilon}

\DeclareMathOperator{\Spec}{Spec}
\DeclareMathOperator{\Cox}{Cox}

\theoremstyle{definition}
\newtheorem*{ack}{Acknowledgements}

\newcommand{\dif}{\mathrm{d}}

\begin{document}

\title[Inhomogeneous quadratic congruences]{Inhomogeneous quadratic
  congruences}

\author{S.\ Baier}
\address{
Mathematisches Institut\\
Universit\"at G\"ottingen\\
Bunsenstr.\ 3--5, 37073\\
G\"ottingen\\ Germany}
\email{sbaier@uni-math.gwdg.de}
\author{T.D.\ Browning}
\address{School of Mathematics\\
University of Bristol\\ Bristol\\ BS8 1TW\\ United Kingdom}
\email{t.d.browning@bristol.ac.uk}

\date{\today}

\begin{abstract}
For given positive integers $a,b,q$ we investigate
the density of solutions $(x,y)\in \ZZ^2$ to congruences 
$
ax+by^2\equiv 0 \bmod{q}
$,  and apply it to detect almost primes on a singular del Pezzo
surface of degree $6$.
\end{abstract}


\maketitle

\section{Introduction}\label{s:intro}

Let $a,b,q$ be non-zero integers with $q\geq 1$ and $(ab,q)=1$.
Let $e,f$ be coprime positive integers with $e\neq f$ 
 and let $X,Y\geq 1$.   A broad array
of problems in number theory can be reduced to estimating the number of solutions  
$(x,y)\in \ZZ^2$ to
congruences of the shape
$$
ax^e+b y^f\equiv 0 \bmod{q},
$$
with $0<x\leq X$ and $0<y\leq Y$.  It is often convenient to focus on
those solutions which  
are coprime to $q$. Let $M_{e,f}(X,Y;a,b,q)$ denote the total number of such solutions. 
A trivial upper bound is given by 
$$
M_{e,f}(X,Y;a,b,q) \ll q^{\ve}\left(\frac{XY}{q} +\min\{X,Y\}\right),
$$
for any $\ve>0$. Here the implied constant is allowed to depend at
most upon the choice of $\ve$, and upon the exponents $e$ and $f$,  
a convention that we adhere to for the remainder of this work. 
One is usually concerned with situations for which either of the ranges $X$ or
$Y$ is substantially smaller than the modulus $q$, where
sharper estimates are sought. 

This paper is inspired by work of Pierce
\cite{pierce}, together with our own recent contribution \cite{baier-browning} to the topic. 
In \cite[Theorem 3]{pierce}, under the assumption that $q$ is square-free  and $\max\{X,2Y\}\leq q$, 
it is shown that 
there is a constant $A=A(e,f)>0$ such that 
\begin{equation} \label{LiPi}
M_{e,f}(X,Y;1,-1,q)
\ll \tau(q)^{A}\left(
\frac{XY}{q} +\frac{X}{\sqrt{q}}
+
\sqrt{q} \log^{2} 2q
\right),
\end{equation}
where $\tau$ is the divisor function.  This estimate is used by Pierce
to obtain a non-trivial bound for the $3$-part $h_{3}(D)$ of the class
number of a quadratic number field $\QQ(\sqrt{D})$, when $|D|$ admits
a divisor of suitable magnitude.  
In \cite{baier-browning} a substantial improvement is obtained when
$(e,f)=(2,3)$ and $q$ is far from being square-free. This in turn is
used to study  the density of elliptic curves with square-free
discriminant and to verify the conjecture of Manin \cite{f-m-t} for
some singular del Pezzo surfaces. 

The above investigations of $M_{e,f}(X,Y;a,b,q)$ use the orthogonality
of additive characters to encode the divisibility condition in the
congruence.  The resulting complete exponential sums can be estimated
using the Weil bound when the modulus is square-free. The present work
is directed at the special case $(e,f)=(1,2)$, wherein the
exponential sums that arise are particularly simple to handle, being
quadratic Gauss sums.  We will establish  the
following refinement of \eqref{LiPi}.

\begin{thm} \label{t:fixedabqX} 
Let $a,b,q$ be non-zero integers with $q\geq 1$ and $(ab,q)=1$ and let $X,Y\geq 1$. 
Then we have
$$
M_{1,2}(X,Y;a,b,q)=
\frac{\varphi(q)}{q^2} \cdot XY +
O\left(\frac{X}{q} \cdot \tau(q)+L(q)\sigma_{-1/2}(q)\left(\frac{Y}{\sqrt{q}}\cdot \tau(q)
+\sqrt{q}L(q)\right)\right), 
$$
where $L(n):=\log (n+1)$, $\sigma_{\alpha}(n):=\sum_{d\mid n}d^{\alpha}$
and $\phi$ is the Euler totient function.
\end{thm}

As an application of this result we will consider the topic of  ``almost primes''' on 
rational surfaces. Later we will produce a version of Theorem \ref{t:fixedabqX} 
in which averaging over the coefficients $a,b,q$ is successfully
carried out and discuss such results in the context of counting
$\QQ$-rational points of bounded height on  singular del Pezzo
surfaces.

Let $X$ be a del Pezzo surface defined over $\QQ$,
embedded in
projective space $\PP^d$, for some $d\geq 3$. 
We may clearly identify $X(\QQ)$ with $X(\ZZ)$, assuming that $X$ is
given by equations with coefficients in $\ZZ$. 
In view of the pioneering work of Bourgain, Gamburd and Sarnak \cite{BGS} one might ask 
whether $X$ has finite ``saturation number'' $r(X(\ZZ),x_0\cdots
x_d)$. 
This is defined to be
the least number $r$ such that the set of
$
\x=(x_0,\ldots,x_d)\in \ZZ^{d+1}
$
for which $[\x]\in X(\ZZ)$ and 
$x_0\cdots x_d$ is a product of at most $r$ primes, is Zariski dense
in $X$.  The investigation of 
Bourgain, Gamburd and Sarnak \cite{BGS}, together with later
refinements of Nevo and Sarnak \cite{NS}, gives effective saturation
numbers for orbits of congruence subgroups of semi-simple
linear algebraic groups in $\mathrm{GL}_n$ defined over $\QQ$. 
In particular these results do not cover del Pezzo
surfaces.

By combining the theory of universal torsors with sieve methods 
it is possible to  demonstrate that 
$r(X(\ZZ),x_0\cdots x_d)<\infty$ for several del Pezzo surfaces. We will
illustrate this line of thought with a particular 
singular del Pezzo surface of degree $6$ over $\QQ$.
Let  $X_0\subset \PP^6$ be such a surface with singularity type
$\mathbf{A}_2$ and both of its $2$ lines defined over $\QQ$. 
Then $X_0$ is given as an intersection of $9$ quadrics in $\PP^6$ and 
Loughran  \cite{loughran} has established the Manin conjecture 
for this surface, together with a power saving in the error term. The
underlying approach involves descending to the universal torsor,
which in this setting is a certain open subset $\mathcal{T}$ of the
affine hypersurface
\begin{equation}
  \label{eq:UT-dp6}
\eta_2\alpha_1^2+\eta_3\alpha_2+\eta_4\alpha_3=0,
\end{equation}
in $\AA^{7}=\Spec\ZZ[\eta_1,\ldots
\eta_4,\alpha_1,\alpha_2,\alpha_3]$.
One is therefore led to count solutions
to this equation in integers $\eta_1, \ldots,\eta_4, \alpha_1,\alpha_2,\alpha_3$, 
subject to a number of constraints. Loughran achieves this be viewing
the equation as a congruence
$\eta_2\alpha_1^2+\eta_3\alpha_2\equiv 0 \bmod{\eta_4}$, for fixed
$\eta_1,\ldots,\eta_4$, before summing the contribution over the
remaining variables. 
We will modify this argument, appealing instead
to Theorem \ref{t:fixedabqX} and the weighted sieve of Diamond and Halberstam \cite{DH},
in order to establish the following result in \S \ref{s:dp6}.

\begin{thm}\label{t:sat}
We have $r(X_0(\ZZ),x_0\cdots x_6)\leq 45$.
\end{thm}

We now turn to the question of averaging the counting function 
$M_{1,2}(X,Y;a,b,q)$  over  suitably constrained values of 
$a,b$ and $q$.   In this endeavour we  
are influenced by the Manin conjecture \cite{f-m-t} 
for singular del Pezzo surfaces $X$ defined over $\QQ$.   
A particularly fruitful approach to this conjecture has two stages:
\begin{itemize}
\item[---] 
one constructs an explicit bijection between rational points of bounded height on $X$
and integral points in a region on a universal torsor $\mathcal{T}_{X}$; and 
\item[---] 
one estimates the number of integral points in this region on the torsor by its volume and shows
that the volume  has the predicted asymptotic growth rate.
\end{itemize}
A geometrically driven approach to 
the first part has been developed  by Derenthal and  Tschinkel \cite[\S 4]{D-T}.
The second part mainly relies on analytic number theory and has been put on a general footing 
by Derenthal \cite{D},  whenever the torsor is a hypersurface.
In this case the torsor equation typically  takes the form 
\begin{equation}\label{eq:genUT}
\alpha_{0}^{a_{0}}
\alpha_{1}^{a_{1}}\cdots 
\alpha_{i}^{a_{i}}+
\beta_{0}^{b_{0}}
\beta_{1}^{b_{1}}\cdots 
\beta_{j}^{b_{j}}+
\gamma_{0}
\gamma_{1}^{c_{1}}\cdots 
\gamma_{k}^{c_{k}}=0,
\end{equation}
with $(a_{0},\ldots, a_{i})\in \NN^{i+1}$, 
$(b_{0},\ldots, b_{j})\in \NN^{j+1}$ and 
$(c_{1},\ldots,c_{k})\in \NN^{k}$.
Work of Hassett \cite[Theorem 5.7]{hassett} shows that 
there is a natural realisation of a  universal torsor as an open
subset via
$\mathcal{T}_{X}\hookrightarrow
\Spec(\Cox(\widetilde{X}))$, where the coordinates of  
$\mathcal{T}_{X}$  correspond to generators of the Cox ring 
of the minimal desingularisation $\widetilde{X}$ of $X$.
Torsor equations such as \eqref{eq:genUT}  are usually 
handled by viewing them as a congruence modulo $q=\gamma_{1}^{c_{1}}\cdots 
\gamma_{k}^{c_{k}}$. Examples of this are provided by Loughran's
treatment of \eqref{eq:UT-dp6}, or by
our work \cite{baier-browning} on 
$M_{2,3}(X,Y;a,b,q)$,  which is pivotal in the resolution of the Manin
conjecture for a singular del Pezzo surface of degree 2.
Experience suggests that there are several examples of
singular del Pezzo surfaces whose torsor equations produce congruences
of the shape
$$
ru^lx + sv^my^2 = 0 \bmod{tw},
$$
for fixed $l,m \in \NN$. A case in point is the cubic surface
with $\mathbf{D}_5$ singularity type which is studied jointly 
by the first author and Derenthal
\cite{browningderenthal}. Here the relevant congruence that emerges is
precisely of this form with 
$l=2$ and $m=1$. Using a result of similar strength to Theorem
\ref{t:fixedabqX} the Manin conjecture is established for this surface
but only with a modest logarithmic saving in the error term.

 Returning to  the behaviour   of
$M_{1,2}(X,Y;a,b,q)$ on average, 
a key feature of the underlying quadratic Gauss sums that  
arise in the proof of 
Theorem \ref{t:fixedabqX}  is that they satisfy explicit formulae. 
This will allow us to study quite general expressions of the form
\begin{equation} \label{Sdef}
\mathcal{S}:=\sum\limits_{(a,b,q)\in S}  c_{a,b,q}
\sum\limits_{\substack{y\in J\\ (y,q)=1}} \
\sum\limits_{\substack{x\in I(a,b,q,y) \\ ax+by^2 \equiv 0 \bmod{q}}} 
1,
\end{equation}
for $c_{a,b,q}\in \CC$.
Here $S\subset \mathbb{Z}^2\times \mathbb{N}$ is a finite set of
triples $(a,b,q)$ such that $(ab,q)=1$,  
$J=\left(y_0,y_0+Y\right]$
is a fixed interval of length $Y\ge 1$, and 
\begin{equation}\label{eq:I}
I(a,b,q,y)=\left(f^-(a,b,q,y),f^+(a,b,q,y)\right]
\end{equation}
is an interval depending on $a,b,q,y$.   
Theorem \ref{t:fixedabqX}  will be an
easy consequence of a general estimate for $\mathcal{S}$, 
which is presented in \S\ref{s:prelim-study}. 
There are two main ingredients at play here: Vaaler's
trigonometric formula for the saw-tooth function  
$\psi(x):=\{x\}-1/2$, where $\{x\}=x-[x]$ denotes the fractional part of $x$, and
the explicit formulae for the quadratic Gauss sum. These will be
recalled in \S \ref{s:prelim}.

When further restrictions are placed on $S$ and $c_{a,b,q}$ one can go further. 
Motivated by our discussion above we set
\begin{equation}\label{eq:S}
S=\left\{(r u^l,s v^m,t w)\ :\ U<u\le 2U, \ V< v\le 2V, \ W<w\le 2W, \ (rs uv,tw)=1\right\}, 
\end{equation}
where $U,V,W\ge 1/2$
and 
$l,m,r,s,t$ are fixed non-zero integers for which $l,m,t\geq 1$
and $(rs,t)=1$.   We shall think of $r,s,t$ as being parameters,
whose dependence we want to keep track of, but $l$ and $m$ are fixed
once and for all.
We further assume that $c_{a,b,q}$ factorises in the form
\begin{equation} \label{coeffcond}
c_{a,b,q}=c_{r u^l,s v^m, t w}=d_{u,v}e_w, \quad \mbox{with $|d_{u,v}|,|e_w|\le 1$.}
\end{equation}
We also entertain the  possibility that  there is a further factorisation
\begin{equation} \label{Duvfact}
d_{u,v}=d_u'\tilde{d}_v, \quad \mbox{with  $|d_u'|, |\tilde{d}_v|\le 1$.} 
\end{equation}
Moreover, we set
$$
\tilde{f}^{\pm}(u,v,w,y):=f^{\pm}(r u^l,s v^m,t w,y).
$$
We make the assumption that 
$\tilde{f}^{\pm}(u,v,w,y)$ are continuous functions 
and have piecewise continuous partial derivatives with respect to the
variables $u,v,w$. We further assume that  
$\tilde{f}^{+}\ge \tilde{f}^{-}$ in the whole domain $(U,2U]\times(V,2V]\times (W,2W]\times J$, with 
\begin{equation} \label{fcondnew}
\left| \frac{\partial^{i+j+k}\tilde{f}^{\pm}}{\partial u^i \partial v^j \partial y^k} (u,v,w,y) \right| \leq \rho^i \sigma^j 
\tau^k F 
\end{equation}
there, for $i,j,k\in \{0,1\}$ such that $i+j+k\not=0$,  
where $\rho,\sigma,\tau,F$ are suitable non-negative numbers. 
For any $H>0$ we set
\begin{equation} \label{Deltadefnew}
\Delta_H=\left(1+\frac{HF\rho U}{tW}\right)\left(1+\frac{HF\sigma V}{tW}\right)\left(1+\frac{HF\tau Y}{tW}\right)
\end{equation}
and
$$
\mathcal{Z}:=
\begin{cases} (tW+U)^{1/2}(tW+V)^{1/2}(UV)^{1/2}W, 
& \mbox{if  \eqref{Duvfact} holds and 
$UV\ge tW$,}\\ 
 (tW)^{1/2}UVW, & \mbox{in general.}
  \end{cases}
$$
We may now record  the outcome of our analysis of the sum $\mathcal{S}$ in \eqref{Sdef} in this setting.

\begin{thm} \label{mainresult} 
Let $\ve>0$ and assume that 
\begin{equation} \label{Hcond}
H\ge \frac{tW}{F}.
\end{equation}
Then under the above hypotheses we have 
$$
\mathcal{S} = \mathop{\sum\limits_{U<u\le 2U}\sum\limits_{V<v\le 2V}\sum\limits_{W<w\le 2W}}_{(rs uv,tw)=1}  
\frac{d_{u,v}e_w}{tw} \sum\limits_{\substack{y_0<y\le y_0+Y\\
    (y,tw)=1}} \tilde{X}(u,v,w,y)+
O\left(\frac{UVWY}{H}\right)+O(\mathcal{T}), 
$$
where 
$\tilde{X}(u,v,w,y):=
\tilde{f}^{+}(u,v,w,y)-\tilde{f}^{-}(u,v,w,y)$ and 
$$
\mathcal{T}:=\Delta_H \left(\frac{Y}{(tW)^{1/2}}\left(U^{1-\{l/2\}}V^{1-\{m/2\}}W+
UVW^{1/2}\right)+\mathcal{Z}\right)(HtUVW)^{\varepsilon}.
$$
\end{thm}

Theorem \ref{mainresult} will be established in \S \ref{proof2}.  The  character sums that arise from the 
 explicit formulae for Gauss sums used in  Theorem \ref{t:fixedabqX} are handled using a mixture of the ordinary large sieve and the large sieve for real characters developed by Heath-Brown
 \cite{heathbrown}.
A review of favourable conditions under which the
main term dominates the error term in Theorem \ref{mainresult} is saved for \S \ref{s:final}.

In line with our discussion of saturation numbers and the Manin conjecture, 
Theorems \ref{t:fixedabqX}  and   
\ref{mainresult}  have significant potential impact in the study of
rational points on del Pezzo surfaces. 
Indeed,
it is likely that  the former result can be used to establish versions
of  Theorem \ref{t:sat}  for   
other
singular del Pezzo surfaces whose universal torsors are open subsets of 
affine  hypersurfaces \eqref{eq:genUT},  which after fixing some of the variables take the basic shape $ax+by^{2}=cz$.   Likewise, the utility of Theorem  
\ref{mainresult}  will be illustrated in
forthcoming work of the first author, where  it
is used to 
establish the Manin conjecture for a further singular cubic surface.

\begin{ack}
While working on this paper the  authors were 
supported by EPSRC grant number \texttt{EP/E053262/1}.  
The authors are grateful to Jianya Liu for drawing their attention to
the question of saturation numbers for del Pezzo surfaces.
\end{ack}

\section{Almost primes on a sextic del Pezzo surface}\label{s:dp6}

We begin by summarising the passage to the universal torsor made
use of by Loughran \cite{loughran} in his resolution of the Manin
conjecture for the split del Pezzo surface $X_0\subset \PP^6$
of degree $6$ with singularity type $\mathbf{A}_2$.   
Working on the Zariski open subset $U\subset X_0$
formed by deleting the lines, it follows from
\cite[Lemma~3.2]{loughran} 
that above each point $[\x]\in U(\ZZ)$, with
$\x=(x_0,\ldots,x_6)$ a primitive integer vector, there is a unique
integral point $(\bet,\bal)$ on the universal torsor $\mathcal{T}$
in \eqref{eq:UT-dp6}, satisfying
$$
\begin{cases}
(\alpha_1,\eta_1\eta_3\eta_4)=
(\alpha_2,\eta_1\eta_2\eta_4)=
(\alpha_3,\eta_1\eta_2\eta_3)=1,\\
(\eta_2,\eta_3)=
(\eta_2,\eta_4)=
(\eta_3,\eta_4)=1,\\
\eta_1,\eta_2,\eta_3,\eta_4>0, \quad \alpha_1\alpha_2\alpha_3\neq 0.
\end{cases}
$$
There is a surjective morphism $\pi: \mathcal{T}\rightarrow X_0$, 
defined over $\QQ$, which is
given by 
\begin{align*}
(\bet,\bal)\mapsto 
&
(\al_2\al_3,\eta^{(1,1,1,0)}\al_1\al_2, 
\eta^{(1,1,0,1)}\al_1\al_3, \\
&\quad \eta^{(2,1,2,1)}\al_2, \eta^{(2,1,1,2)}\al_3,
\eta^{(4,2,3,3)}, \eta^{(3,2,2,2)}\al_1),
\end{align*}
where $\eta^{(a,b,c,d)}=\eta_1^a\eta_2^b\eta_3^c\eta_4^d$.
In particular one notes that 
$$
x_0\cdots x_6 = 
\eta_1^{13}\eta_2^{8}\eta_3^{9}
\eta_4^9
\al_1^3\al_2^3\al_3^3,
$$
under $\pi$.

In order to establish Theorem \ref{t:sat} it will suffice to produce
a Zariski dense set of almost prime solutions of the torsor
equation. If one restricts to points $x=[\x]\in U(\ZZ)$ with anticanonical
height $H(x)\leq B$ then one gets corresponding 
size restrictions on the integral
points $(\bet,\bal)$ via $\pi.$
Since we are merely concerned with a lower bound for the
associated counting function, we may freely specialise
convenient constraints on the torsor variables at the outset. With
this in mind we will only consider solutions in which
$\eta_1=\eta_2=\eta_3=1$ and $\eta_4$ is prime.
It would be tempting to set further variables equal to unity in the
torsor equation,
but 
one easily demonstrates that such points do not constitute a Zariski
dense open subset of $X_0$.

For any $t\geq 1$, let $M_t(B)$ denote the number of
$(\eta_4,\alpha_1,\alpha_2,\alpha_3)\in \ZZ^4$ such that 
$$
\begin{cases}
(\al_1\al_2,\eta_4)=1, \quad \eta_4>0, \quad
\alpha_1\alpha_2\alpha_3\neq 0,\\
\max\{
|\al_2\al_3|,|\al_1\al_2|, |\eta_4\al_1\al_3|, |\eta_4 \al_2|, |\eta_4^{2}\al_3|,
|\eta_4^{3}|, |\eta_4^2\al_1|\}\leq B,
\end{cases}
$$
and 
$
\alpha_1^2-\alpha_2+\eta_4\alpha_3=0,
$
with $\eta_4$ being prime and 
$\alpha_1\alpha_2\alpha_3=P_t$, where
$n=P_t$ means that $n$ has at most $t$ prime factors. 
If $N_r(B)$  denotes the number of $x=[\x]\in U(\ZZ)$ for which
$H(x)\leq B$ and $x_0\cdots x_6=P_r$, then
it is clear that
$$
N_r(B)\geq \frac{1}{2} M_{r/3-3}(B),
$$
if $r\geq 12$. In view of \cite[Theorem~1.1]{loughran} one has $N_r(B)\ll
B\log^3 B$ for any $r\geq 1$. Hence, in order to establish Theorem
\ref{t:sat}, it will suffice to establish the existence of absolute
constants $t\geq 1$ and $k\in \ZZ$ for which 
\begin{equation}\label{eq:needed}
M_t(B)\gg B \log^{k}B.
\end{equation}
In fact we shall demonstrate that the choices $t=12$ and $k=-5$ are
permissible in this estimate, which will therefore  terminate the 
proof of Theorem~\ref{t:sat}. 

It is clear that $0<\eta_4\leq B^{1/3}$ in any point counted by 
$M_t(B)$. In estimating $M_t(B)$ from below it will be convenient to
only consider primes $\eta_4$ in the range $\frac{1}{2}B^{1/3}<\eta_4\leq
B^{1/3}$. Likewise we will insist that 
\begin{equation}
  \label{eq:heights-dp6}
0<\al_1\leq \frac{1}{2}B^{1/3}, \quad 
0<\al_2\leq \frac{1}{2}B^{2/3}.
\end{equation}
Together with the equation 
$\alpha_1^2-\alpha_2+\eta_4\alpha_3=0$,
these restrictions on $\eta_4,\alpha_1,\al_2$ 
ensure that the size restrictions on $\alpha_3$ hold automatically,
apart from the condition that $\alpha_3\neq 0$. 
For any prime $q$ let 
$L_{t}(B;q)$ denote the number $(\al_1,\al_2,\al_3)\in \ZZ^3$ for which 
\eqref{eq:heights-dp6} holds and 
$(\al_1\al_2,q)=1$, 
with
\begin{equation}
  \label{eq:q-torsor}
\alpha_1^2-\alpha_2+q\alpha_3=0
\end{equation}
and $\alpha_1\alpha_2\alpha_3=P_t$.
In particular we have
\begin{equation}
  \label{eq:123}
  |\al_1\al_2\al_3|\leq B^{4/3},
\end{equation}
for any point counted by $L_{t}(B;q)$. 
We now have the inequality
\begin{equation}
  \label{eq:fix-eta4}
  M_t(B)\geq \sum_{\substack{\frac{1}{2}B^{1/3}<q\leq
B^{1/3}\\ q \mbox{\scriptsize{ ~prime}} }} L_{t}(B;q) +O(B^{2/3}),
\end{equation}
since points with $\alpha_3=0$ trivially contribute $O(B^{2/3})$.
The stage is now set for an application of sieve methods to
estimate $L_t(B;q)$ from below.

Our work will make use of the weighted sieve of dimension $\kappa>1$,
as developed by Diamond and Halberstam \cite[Chapter 11]{DH}.
We  recall here the basic set-up. 
Given a finite sequence $\cA=\{a_n\}_{n\geq 1}$ of non-negative
real numbers, 
the weighted sieve can be used to
determine a precise lower bound for the sum
$$
S_t(\cA) = \sum_{\substack{n=P_t}} a_n.
$$
We proceed to record the basic sieve assumptions.

{\em Condition ($W_0$):}
There exists an approximation $X$ to $\sum_{n\geq 1} a_n$,
such that for any square-free $d \in \NN$ we have 
$$
\sum_{\substack{n\geq 1\\ n\equiv 0 \bmod{d}}}a_n
= \frac{\rho(d)}{d}X+R_d(\cA),
$$
where $\rho$ is a multiplicative
function satisfying $\rho(1)=1$ and 
$$
0\leq \rho(p)<p,
$$
for any prime $p$.

{\em Condition ($W_1$):}
$\cA$ has dimension $\kappa>1$, by which we mean that there exists
 $c_1>0$ such that 
$$
\prod_{w\leq p\leq z} \left(1-\frac{\rho(p)}{p}\right)^{-1}\leq 
\left(\frac{\log z}{\log w}\right)^\kappa \left(1+\frac{c_1}{\log w}\right), 
$$
for any $2\leq w\leq z$. 

{\em Condition ($W_2$):}
$\cA$ has level of distribution $\tau\in (0,1)$, by which we mean that there
exists $c_2\geq 1$ and $c_3\geq 2$ such that 
$$
\sum_{d\leq X^\tau \log^{-c_2}X} \mu^2(d)4^{\omega(d)}|R_d(\cA)| \leq
c_3 \frac{X}{\log^{\kappa+1}X}, 
$$
where $\omega(d)$ denotes the number of prime factors of $d$. 

Assume conditions ($W_0$), ($W_1$) and ($W_2$). Let $\mu$ be a
constant such that 
\begin{equation}
  \label{eq:W3}
\max_{a_n \in \cA} n \leq X^{\tau \mu}.  
\end{equation}
Then it follows from
\cite[Section 11.4]{DH} that there exists a
real constant $\beta_\kappa>1$ such that 
\begin{equation}
  \label{eq:sieve}
S_t(\cA) \gg X \prod_{p<X^{\tau/(2\beta_\kappa-1)}} \left(1-\frac{\rho(p)}{p}\right),
\end{equation}
provided that
$
t>
\mu-1+(\mu-\kappa)(1-1/\beta_\kappa)+(\kappa+1)\log \beta_\kappa.
$
The values of the sieving parameters 
$\beta_\kappa$ are tabulated in \cite[Chapter 17]{DH}.

\bigskip

For a fixed prime $q$ 
satisfying
\begin{equation}\label{eq:q-size}
\frac{1}{2}B^{1/3}<q\leq  B^{1/3},
\end{equation}
we take $\cA$ to be set of $a_n=a_{n}(B;q)$, where each $a_n$ is 
the cardinality of $(\al_1,\al_2,\al_3)\in \ZZ^3$ for which 
\eqref{eq:heights-dp6} and \eqref{eq:q-torsor} hold, with
$(\al_1\al_2,q)=1$ and 
$\alpha_1\alpha_2\alpha_3=\pm n$. 
In particular it is clear that 
$S_t(\cA)=L_t(B;q)$ and 
we may take
\begin{equation}
  \label{eq:XX}
X=\frac{\phi(q)B}{4q^2}=\frac{\phi^*(q)B}{4q},
\end{equation}
where $\phi^*(n)=\phi(n)/n$, 
since $\sum_{n\geq 1}a_n$ is asymptotically equal to $\frac{\phi(q)}{4q^2}B$
as $B\rightarrow \infty$. 
For any square-free $d\in \NN$, it follows from the inclusion--exclusion principle 
that
\begin{equation}\label{eq:in-ex}
\sum_{\substack{n\geq 1\\ n\equiv 0 \bmod{d}}} a_n
=\mu(d)\sum_{\substack{\mathbf{e}\in \NN^3\\
p\mid e_1e_2e_3 \Leftrightarrow p\mid d
}}
\mu(e_1)\mu(e_2)\mu(e_3)\#S_{\mathbf{e}}(\cA),
\end{equation}
where
$S_{\ma{e}}(\cA)$ denotes the set of 
$(\al_1,\al_2,\al_3)\in \ZZ^3$ for which 
\eqref{eq:heights-dp6} holds and 
$(\al_1\al_2,q)=1$, with \eqref{eq:q-torsor} and $e_i \mid \al_i$ for
$1\leq i \leq 3$. In particular we will only be interested in
$\ma{e}\in \NN^3$ for which $(e_1e_2,q)=1$. 
Making an obvious change of variables we deduce that
$S_{\ma{e}}(\cA)$ is the set of 
$(\beta_1,\beta_2)\in \ZZ^2$ for which 
$$
0<\beta_1\leq \frac{B^{1/3}}{2e_1}, \quad 
0<\beta_2\leq \frac{B^{2/3}}{2e_2},
$$
with $(\beta_1\beta_2,q)=1$ and  
$$
e_1^2\beta_1^2-e_2\beta_2\equiv 0\bmod{e_3q}.
$$
We need to remove common factors of $e_i\beta_i$ with $e_3$. 
Let $k=(e_1,e_2,e_3)$ and write $e_i=ke_i'$. In particular
$(k,e_1'e_2'e_3')=1$ since $e_1,e_2,e_3$ are square-free.
The above congruence then becomes
$$
ke_1'^2\beta_1^2-e_2'\beta_2\equiv 0\bmod{e_3'q}.
$$
We now  put  $k_{i,3}=(e_i',e_3')$ for $i=1,2$. Then $k_{2,3}$ divides
$\beta_1$ and $k_{1,3}$ divides $\beta_2$. Making the obvious changes
of variables we 
see that
$
S_{\ma{e}}(\cA)$
is the set of $(\beta_1',\beta_2')\in \ZZ^2$ for which 
$$
0<\beta'_1\leq \frac{B^{1/3}}{2e_1k_{2,3}}, \quad 
0<\beta'_2\leq \frac{B^{2/3}}{2e_2k_{1,3}},
$$
with $(\beta'_1\beta_2',q)=1$ and  
$$
kk_{1,3}k_{2,3}f_1^2\beta_1'^2-f_2\beta_2'\equiv 0\bmod{f_3q},
$$
where
$$
f_1=\frac{e_1}{kk_{1,3}}, \quad
f_2=\frac{e_2}{kk_{2,3}}, \quad
f_3=\frac{e_3}{kk_{1,3}k_{2,3}}.
$$
Finally we need to remove common factors of $\beta_1', \beta_2', f_3$.
Let $\ell=(f_3,\beta_1',\beta_2')$. Making a suitable change of
variables, we now have 
\begin{equation}\label{eq:sund}
\#S_{\ma{e}}(\cA)=\sum_{\ell\mid f_3}\#S_{\ma{e},\ell}(\cA),
\end{equation}
where $S_{\ma{e},\ell}(\cA)$
is the set of $(x,y)\in \ZZ^2$ for which 
$$
0<x \leq \frac{B^{2/3}}{2e_2k_{1,3}\ell}=X_0, \quad 
0<y \leq \frac{B^{1/3}}{2e_1k_{2,3}\ell}=Y_0, 
$$
say, with $(x y,\widetilde{q})=1$ and  
$$
\widetilde{a}x+\widetilde{b} y^2\equiv 0\bmod{\widetilde{q}},
$$
where
$$
\widetilde{a}=-f_2, \quad
\widetilde{b}=kk_{1,3}k_{2,3}f_1^2\ell,
\quad
\widetilde{q}=\frac{f_3q}{\ell}=\frac{e_3q}{kk_{1,3}k_{2,3}\ell}.
$$
In particular we have $\widetilde{q}\geq 1$ and 
$(\widetilde{a}\widetilde{b},\widetilde{q})=1$ in this counting problem.

We appeal to Theorem \ref{t:fixedabqX} to estimate
$\#S_{\ma{e},\ell}(\cA)$ for given $\ma{e}\in \ZZ^3$. 
The main term is 
\begin{align*}
\frac{\phi(\widetilde{q})X_0Y_0}{\widetilde{q}^2}
&=
\frac{B}{4e_1e_2k_{1,3}k_{2,3}\ell^2}\cdot \phid\left(
\frac{e_3q}{kk_{1,3}k_{2,3}\ell}\right)
\cdot
\frac{kk_{1,3}k_{2,3}\ell}{e_3q}\\
&= X\cdot
\frac{1}{e_1e_2e_3\ell}\cdot \phid\left(
\frac{e_3q}{kk_{1,3}k_{2,3}\ell}\right)
\cdot
\frac{k}{\phid(q)},
\end{align*}
where $X$  is given by \eqref{eq:XX}.
Likewise the error terms are seen to contribute
\begin{align*}
\ll 
\widetilde{q}^{\ve}\left(\frac{X_0}{\widetilde{q}}+ \frac{Y_0}{\sqrt{\widetilde{q}}}+\sqrt{\widetilde{q}}\right)
&\ll (dq)^{\ve}\left(\frac{B^{2/3}}{q}+\frac{B^{1/3}}{\sqrt{q}}+\sqrt{dq}\right)\\
&\ll 
d^\ve B^{1/3+\ve}+d^{1/2+\ve}B^{1/6+\ve},
\end{align*}
for any $\ve>0$, 
since $q$  is assumed to be in the range \eqref{eq:q-size}.

Employing \eqref{eq:in-ex} and \eqref{eq:sund} 
we now obtain
\begin{align*}
\sum_{\substack{n\geq 1\\ n\equiv 0 \bmod{d}}} a_n
&=\frac{\rho(d)}{d}X +R_d(\cA),
\end{align*}
with $R_d(\cA)=O(d^\ve B^{1/3+\ve} + d^{1/2+\ve}B^{1/6+\ve})$ and
\begin{align*}
\rho(d)
&=\mu(d)d\sum_{\substack{\mathbf{e}\in \NN^3\\
p\mid e_1e_2e_3 \Leftrightarrow p\mid d\\
(e_1e_2,q)=1
}}
\frac{\mu(e_1)\mu(e_2)\mu(e_3)}{e_1e_2e_3}
\sum_{\ell\mid f_3}\frac{1}{\ell}\cdot \phid\left(
\frac{f_3q}{\ell}\right)
\cdot
\frac{k}{\phid(q)}\\
&=\mu(d)d\sum_{\substack{\mathbf{e}\in \NN^3\\
p\mid e_1e_2e_3 \Leftrightarrow p\mid d\\
(e_1e_2,q)=1
}}
\frac{\mu(e_1)\mu(e_2)\mu(e_3)(e_1,e_2,e_3)}{e_1e_2e_3}
\sum_{\ell\mid f_3}\frac{1}{\ell}\cdot \frac{\phid(f_3/\ell)}{\phid((f_3/\ell,q))},
\end{align*}
where we recall that 
$$
k=(e_1,e_2,e_3),  \quad
k_{i,3}=\left(\frac{e_i}{k},\frac{e_3}{k}\right), \quad
f_3=\frac{e_3}{kk_{1,3}k_{2,3}},
$$
for $i=1,2$.
In particular $\rho(d)$ is a multiplicative arithmetic function of $d$. One easily calculates that 
$\rho(q)=1+1/q$ and 
$$
\rho(p)=-p\left(-\frac{3}{p}+\frac{2}{p^2}\right)=
3-\frac{2}{p}
$$
if $p\neq q$.
It is now clear that all the hypotheses of conditions ($W_0$) and ($W_1$) in the sieve are satisfied, with 
$\kappa=3$ and $c_1>0$ a suitable  absolute constant. 
In view of \eqref{eq:q-size} and \eqref{eq:XX}, we have
$$
\frac{B^{2/3}}{\log\log B}\ll X\ll B^{2/3}.
$$
Hence we deduce that
\begin{align*}
\sum_{d\leq X^\tau} |R_d(\cA)| 
&\ll
 X^{\tau(1+\ve)} B^{1/3+\ve} + X^{\tau(3/2+\ve)}B^{1/6+\ve}\\
&\ll
 X^{1/2+\tau+2\ve}  + X^{1/4+3\tau/2+2\ve},
\end{align*}
whence 
condition ($W_2$) is satisfied for any $\tau<1/2$, with $c_2=1$ and
suitable $c_3=c_3(\ve)\geq 2$.  
Moreover, in view of \eqref{eq:123}, it is clear that we may take any 
$\mu>4$ in \eqref{eq:W3}.

Our efforts up to this point  justify taking 
$$
\kappa=3, \quad \mu>4, \quad \tau>\frac{1}{2}
$$
in the sieve assumptions. 
We thus arrive at the lower bound \eqref{eq:sieve} for
$S_t(\cA)=L_t(B;q)$, provided that
$$
t>
4-1/\beta_3+4\log \beta_3.
$$
For the choice $\kappa=3$ it follows from the tabulation of sieving
limits in Diamond and Halberstam \cite[Table 17.1]{DH} that 
$
\beta_3=6.640859. 
$
Hence we may take $t\geq 12$ in \eqref{eq:sieve}, with which choice one
has
$$
L_{t}(B;q) \gg \frac{B^{2/3}}{\log^3 B \log\log B}\gg
\frac{B^{2/3}}{\log^4 B},
$$
uniformly in $q$.  Once inserted into 
  \eqref{eq:fix-eta4} and combined with the prime number theorem,  
this therefore establishes the lower bound 
for $M_t(B)$ in \eqref{eq:needed} with $t=12$ and $k=-5$, as required
to   
complete the proof of Theorem \ref{t:sat}.

\section{Technical tools}\label{s:prelim}

In this section we  collect together the technical lemmas that will feature in our proof of Theorems 
\ref{t:fixedabqX} and \ref{mainresult}.  We will use the following
approximation of the function $\psi(x)$ using trigonometric
polynomials due to Vaaler (see Graham and Kolesnik \cite[Theorem
A.6]{grahamkolesnik}, for example).

\begin{lemma}\label{l:v}
Let $H>0$. Then 
there exist coefficients $a_{h}\in \RR$ satisfying $a_{h}\ll 1/|h|$,  such that 
$$
\left|
\psi(x)-
\sum\limits_{1\le |h|\le H} a_{h} \e(hx)\right| \leq
\frac{1}{H+1}  \sum\limits_{|h|\le H} \left(1-\frac{|h|}{H+1}\right) \e(hx).
$$
\end{lemma}


This result will lead to the intervention of  exponential sums,
which once evaluated will also produce certain types of character
sums. To handle these we will require the following 
variant of Heath-Brown's large sieve for real characters \cite[Corollary 4]{heathbrown}.

\begin{lemma}\label{lsreal} Let $\ve>0$, let $M, N\in \NN$, and let $a_1,...,a_M$ and $b_1,...,b_N$ be arbitrary complex numbers
satisfying $|a_m|$, $|b_n|\le 1$. Then
$$
\sum\limits_{\substack{m\le M\\ (m,2)=1}} \sum\limits_{n\le N} a_mb_n
\left(\frac{n}{m}\right) \ll
(MN)^{\varepsilon}\left(MN^{1/2}+M^{1/2}N\right). 
$$
\end{lemma}

We end this section with an explicit evaluation of the quadratic Gauss sums
\begin{equation} \label{gaussumdef}
\mathcal{G}(s,t;u):=\sum\limits_{n=1}^u \e\left(\frac{sn^2+tn}{u}\right),
\end{equation}
for given non-zero integers $s,t,u$ such that $u\geq 1$.
Let 
$$
\delta_n:=\begin{cases} 0, & \mbox{ if }n\equiv 0\bmod{2}, \\ 1 & \mbox{ if } n\equiv 1\bmod{2}, \end{cases}
\quad
\epsilon_n:=\begin{cases} 1, & \mbox{ if }n\equiv 1\bmod{4}, \\ i, & \mbox{ if } n\equiv 3\bmod{4}.
\end{cases}
$$
The next lemma gives the value of $\mathcal{G}(s,t;u)$ if $(s,u)=1$.

\begin{lemma} \label{gaussevlemma} 
Suppose that $(s,u)=1$. Then we have the following.
\begin{itemize}
\item[(i)] If $u$ is odd, then
\begin{equation} \label{gaussev}
\mathcal{G}(s,t;u)=\epsilon_u \sqrt{u}  \left(\frac{s}{u}\right) \e\left(-\frac{\overline{4s}t^2}{u}\right).
\end{equation}
\item[(ii)] If $u=2v$ with $v$ odd, then
\begin{equation} \label{gaussev3}
\mathcal{G}(s,t;u)=2  \delta_t \epsilon_v\sqrt{v} \left(\frac{2s}{v}\right) \e\left(-\frac{\overline{8s}t^{2}}{v}\right).
\end{equation}
\item[(iii)] If $4\mid u$, then
\begin{equation} \label{gaussev2}
\mathcal{G}(s,t;u)=(1+i) \varepsilon_s^{-1} (1-\delta_t) \sqrt{u}\left(\frac{u}{s}\right) \e\left(-\frac{\overline{s}t^{2}}{4u}\right).
\end{equation}
\end{itemize}
\end{lemma}

\begin{proof} (i) Let $u$ be odd and assume $(s,u)=1$. Then, by Lemmas 3 and 9 in \cite{estermann}, we have
$$
\mathcal{G}(s,t;u)=\e\left(-\frac{\overline{4s}t^2}{u}\right) \left(\frac{s}{u} \right) \mathcal{G}(1,0;u). 
$$
Gauss proved (see Nagell \cite[Theorem 99]{nagell}, for example) that
\begin{equation} \label{gauss}
\mathcal{G}(1,0;n)=\begin{cases} 
(1+i)\sqrt{n},  & \mbox{ if } n \equiv 0 \bmod 4, \\ 
\sqrt{n},  & \mbox{ if } n \equiv 1 \bmod 4, \\  
0,  & \mbox{ if } n \equiv 2 \bmod 4, \\
i\sqrt{n},  & \mbox{ if } n \equiv 3 \bmod 4,
\end{cases}
\end{equation}
from which \eqref{gaussev} follows.

(ii) Let $2\|u$ and assume $(s,u)=1$. Write $u=2v$ and note that
$2\nmid v$. If $2\mid t$ then 
$$
\mathcal{G}(s,t;2v)=\e\left(-\frac{\overline{s}t^2}{4u}\right)\mathcal{G}(s,0;2v)=0
$$
by Lemmas 4 and 9 in \cite{estermann}. If $2\nmid t$, then 
$$
\mathcal{G}(s,t;2v)=2\e\left(-\frac{\overline{8s} t^2}{v}\right)\mathcal{G}(2s,0;v)
$$
by Lemma 6 in \cite{estermann}. Now applying \eqref{gaussev} gives \eqref{gaussev3}.

(iii) Let $4\mid u$ and assume $(s,u)=1$. If $2\nmid t$, then $\mathcal{G}(s,t;u)=0$ by Lemma 5 in \cite{estermann}. Assume that $2\mid t$. Then, by Lemma 4 in \cite{estermann}, we have
$$
\mathcal{G}(s,t;u)=\e\left(-\frac{\overline{s}t^2}{4u}\right)\mathcal{G}(s,0;u).
$$
For $(s,u)=1$, the Gauss sum satisfies the reciprocity law 
$$
\mathcal{G}(s,0;u)\mathcal{G}(u,0;s)=\mathcal{G}(1,0;su).
$$
Noting that $s$ is odd and $4\mid su$, and applying \eqref{gaussev} to $\mathcal{G}(u,0;s)$ and \eqref{gauss} to $\mathcal{G}(1,0;su)$, we deduce \eqref{gaussev2}.
\end{proof}

\section{Analysis  of $\mathcal{S}$}\label{s:prelim-study}

In this section we begin in earnest our investigation of the sum $\cS$ presented in \eqref{Sdef}.
Recall that $c_{a,b,q}$ are arbitrary complex numbers and 
 $S\subset \mathbb{Z}^2\times \mathbb{N}$ is a finite set of triples $(a,b,q)$ such that $(ab,q)=1$, 
with 
$J:=\left(y_0,y_0+Y\right]$
and $I(a,b,q,y)$ given by \eqref{eq:I}, respectively.
We henceforth stipulate that 
$$
\mbox{domain}(f^+)=\mbox{domain}(f^-)=\mathcal{R},
$$
where 
\begin{equation} \label{cubo}
\mathcal{R}=(a_0,a_0+A]\times (b_0,b_0+B]\times (q_0,q_0+Q] \times (y_0,y_0+Y]
\end{equation}
is a half-open cuboid in $\mathbb{R}^4$ such that $S\times J\subset 
\mathcal{R}$.
We further suppose that $f^{\pm}(a,b,q,y)$ are continuous, have 
piecewise continuous partial derivatives with respect to the variables $a,b,y$, 
and 
satisfy $f^{+}\ge f^{-}$ in the whole domain $\mathcal{R}$. Moreover, we set
$$
X(a,b,q,y):=|I(a,b,q,y)|=f^+(a,b,q,y)-f^-(a,b,q,q).
$$

Our first step is rewrite the congruence 
$ax+by^2 \equiv 0 \bmod{q}$ in $\cS$ as
$$
x+\overline{a}by^2\equiv 0 \bmod{q},
$$
where  $\overline{a}$ denotes the multiplicative inverse of $a$ modulo $q$. 
It follows that
\begin{align*} 
 \sum\limits_{\substack{x\in I(a,b,q,y) \\ ax+by^2 \equiv 0 \bmod{q}}} 
 \hspace{-0.3cm}
 1  &=
 \left[\frac{f^+(a,b,q,y)}{q}+\frac{\overline{a}by^2}{q}\right]-
\left[\frac{f^-(a,b,q,y)}{q}+\frac{\overline{a}by^2}{q}\right] \\
&= \frac{X(a,b,q,y)}{q} - \psi\left(\frac{f^+(a,b,q,y)}{q}+\frac{\overline{a}by^2}{q}\right)+
\psi\left(\frac{f^-(a,b,q,y)}{q}+\frac{\overline{a}by^2}{q}\right).
\end{align*}
We may therefore write
 \begin{equation} \label{decomp}
 \mathcal{S}=\mathcal{M}-\mathcal{E}^{+}+\mathcal{E}^{-},
 \end{equation}
 where
 \begin{equation} \label{mainterm}
 \mathcal{M}:=\sum\limits_{(a,b,q)\in S}  \frac{c_{a,b,q}}{q} \sum\limits_{\substack{y\in J\\ (y,q)=1}}   X(a,b,q,y)
 \end{equation}
 is the main term and
$$
\mathcal{E}^{\pm}:=\sum\limits_{(a,b,q)\in S}  c_{a,b,q} \sum\limits_{\substack{y\in J\\ (y,q)=1}} 
\psi\left(\frac{f^{\pm}(a,b,q,y)}{q}+\frac{\overline{a}by^2}{q}\right)
$$
are error terms.  The next result  is an easy consequence of 
Lemma \ref{l:v} and  transforms these error terms into exponential sums.

\begin{lemma} 
Let $H>0$. Then we have
$\left| \mathcal{E}^{\pm} \right| \ll \mathcal{E}+\mathcal{F}^{\pm},
$
where 
\begin{align}\label{Eerror}
\mathcal{E}
&:=\frac{Y}{H} \sum\limits_{(a,b,q)\in S} |c_{a,b,q}|,\\
\label{Ferror}
\mathcal{F}^{\pm}&:=\sum\limits_{1\le h\le H} \frac{1}{h} \left| \sum\limits_{(a,b,q)\in S} 
C_{a,b,q}S_h^{\pm}(a,b,q)\right|,
\end{align}
with $C_{a,b,q}:=c_{a,b,q}+|c_{a,b,q}| $ and 
$$
S_h^{\pm}(a,b,q):=\sum\limits_{\substack{y\in J\\ (y,q)=1}} \\\e\left(h\cdot \frac{f^{\pm}(a,b,q,y)}{q}\right) \\\e\left(h\cdot \frac{\overline{a}by^2}{q}\right).
$$
\end{lemma}

We proceed to reduce our exponential sums $S_h^{\pm}(a,b,q)$ to complete quadratic Gauss sums.
First we remove the factor $\e\left(h\cdot f^{\pm}(a.b,q,y)/q\right)$ using partial summation, obtaining
\begin{align*} 
S_h^{\pm}(a,b,q) =~& \e\left(h\cdot \frac{f^{\pm}(a,b,q,y_0+Y)}{q}\right)T_h(a,b,q,y_0+Y)\\  &-
\frac{2\pi ih}{q}\int\limits_{y_0}^{y_0+Y} \left(\frac{\partial}{\partial t} f^{\pm}(a,b,q,t)\right) \e\left(h\cdot \frac{f^{\pm}(a,b,q,t)}{q}\right) T_h(a,b,q,t) \dif t,
\end{align*}
where
$$
T_h(a,b,q,t):=\sum\limits_{\substack{y_0<y\le t\\ (y,q)=1}} \e\left(h\cdot \frac{\overline{a}by^2}{q}\right).
$$
Next we remove the coprimality condition $(y,q)=1$ using M\"obius inversion, getting
$$
T_h(a,b,q,t):=\sum\limits_{e\mid q} \mu(e) \sum\limits_{y_0/e<y\le t/e} \e\left(he\cdot \frac{\overline{a}by^2}{q/e}\right).
$$
We remove common factors by writing
\begin{equation} \label{primes}
q'=\frac{q/e}{(he,q/e)}, \quad h'=\frac{he}{(he,q/e)}
\end{equation}
and observing that
$$
T_h(a,b,q,t)=\sum\limits_{e\mid q} \mu(e) \sum\limits_{y_0/e<y\le t/e} \e\left(\frac{h'\overline{a}by^2}{q'}\right), 
$$
with  $(h',q')=1$. 
Here we note that $q'$ and $h'$ depend on $e$, $q$ and $h$. The inner
sum 
is an incomplete quadratic Gauss sum which we complete by writing
\begin{align*}
\sum\limits_{y_0/e<y\le t/e} \e\left(\frac{h'\overline{a}by^2}{q'}\right) & =\sum\limits_{n=1}^{q'} \e\left(\frac{h'\overline{a}bn^2}{q'}\right) \cdot \frac{1}{q'}\cdot \sum\limits_{k=1}^{q'} 
\sum\limits_{y_0/e<l\le t/e} \e\left(k\cdot \frac{n-l}{q'}\right)\\
& = \frac{1}{q'} \cdot \sum\limits_{k=1}^{q'}  r_e(k,q';t) \mathcal{G}(h'\overline{a}b,k;q'),
\end{align*}
where $\mathcal{G}(h'\overline{a}b,k;q')$ is given by  \eqref{gaussumdef} and
$$
r_e(k,q';t):=\sum\limits_{y_0/e<l\le t/e} \e\left(-\frac{kl}{q'}\right) \ll
\min\left\{Y/e,\|k/q'\|^{-1}\right\},
$$
if  $y_0\le t\le y_0+Y$.

Let
$$
g^{\pm}_h(a,b,q,t):=\left(\frac{\partial}{\partial t} f^{\pm}(a,b,q,t)\right)\e\left(h\cdot \frac{f^{\pm}(a,b,q,t)}{q}\right).
$$
Our work so far has shown that 
\begin{align*}
S_h^{\pm}(a,b,q) = ~& \sum\limits_{e\mid q} \frac{\mu(e)}{q'}\cdot \sum\limits_{k=1}^{q'} \mathcal{G}(h'\overline{a}b,k;q')B(e,k),
\end{align*}
with 
\begin{align*}
B(e,k):=
 \e\left(h\cdot \frac{f^{\pm}(a,b,q,y_0+Y)}{q}\right) r_e(k,q';y_0+Y) 
-  \frac{2\pi ih}{q}\int\limits_{y_0}^{y_0+Y} g^{\pm}_h(a,b,q,t)r_e(k,q';t) \dif t.
\end{align*}

Returning to the error terms $\mathcal{F}^{\pm}$ in \eqref{Ferror}, we
deduce that 
$$
\mathcal{F}^{\pm}
\ll  \sum\limits_{h\le H} \sum\limits_{q} \sum\limits_{e\mid q} 
\frac{1}{hq'} \sum\limits_{k=1}^{q'}
\min\left\{Y/e,\|k/q'\|^{-1}\right\}\left(R_1(e,h,q,k)+R_2(e,h,q,k)\right), 
$$
with
\begin{align*}
R_1(e,h,q,k)
&:= \left| 
\sum\limits_{\substack{a,b\\ (a,b,q)\in S}} C_{a,b,q}\mathcal{G}(h'\overline{a}b,k;q')
\e\left(h\cdot \frac{f^{\pm}(a,b,q,y_0+Y)}{q}\right) \right|,\\
R_2(e,h,q,k)&:= \frac{h}{q} \int\limits_{y_0}^{y_0+Y} \left| 
\sum\limits_{\substack{a,b\\ (a,b,q)\in S}} C_{a,b,q}\mathcal{G}(h'\overline{a}b,k;q')
g_{h}^{\pm}(a,b,q,t) \right| \dif t. 
\end{align*}

Now we are ready to evaluate $R_1$ and $R_2$ using the formulae for
Gauss sums in Lemma~\ref{gaussevlemma}.  
Since we get slightly different formulae in the cases (i), (ii), (iii), it is reasonable to break the term 
on the right-hand side of our estimate for $\mathcal{F}^{\pm}$ into $\mathcal{F}_1^{\pm}$, $\mathcal{F}_2^{\pm}$ and $\mathcal{F}_4^{\pm}$, where $\mathcal{F}_1^{\pm}$ denotes the 
contribution of odd moduli $q'$,  $\mathcal{F}_2^{\pm}$ denotes the contribution of moduli with $2\|q'$, and $\mathcal{F}_4^{\pm}$ denotes the 
contribution of moduli with $4\mid q'$. For $i=1,2,4$, we define 
$$
\xi_i(q'):=\begin{cases} 
1, & \mbox{ if $i=1$ and $q'$ is odd,}\\ 
1, & \mbox{ if $i=2$ and $2\|q'$,} \\ 
1, & \mbox{ if $i=4$
and $4\mid q'$,} \\ 
0, & \mbox{otherwise}.      
\end{cases}
$$
We may therefore write
\begin{equation} \label{L1def}
\mathcal{F}_i^{\pm}= \sum\limits_{h\le H} \sum\limits_{q}  \sum\limits_{e\mid q} 
\frac{\xi_i(q')}{hq'} \sum\limits_{k=1}^{q'} \min\left\{Y/e,\|k/q'\|^{-1}\right\}\left(R_1(e,h,q,k)+R_2(e,h,q,k)\right),
\end{equation}
for $i=1,2,4$.

For brevity, we only evaluate $R_1$ and $R_2$ when $q'$ is odd, which is the relevant case for the treatment of
$\mathcal{F}_1^{\pm}$. The cases $2\|q'$ and $4\mid q'$ can each be handled similarly. 
If $(q',2h')=1$, then Lemma \ref{gaussevlemma}(i) yields
$$
\mathcal{G}(h'\overline{a}b,k;q')=\epsilon_{q'} \sqrt{q'} \cdot \left(\frac{h'ab}{q'}\right) \e\left(-\frac{\overline{4bh'}\cdot ak^2}{q'}\right).
$$
Hence, in this case we have
\begin{equation} \label{R1}
R_1(e,h,q,k)= \sqrt{q'} \left| 
\sum\limits_{\substack{a,b\\ (a,b,q)\in S}} C_{a,b,q} \left(\frac{ab}{q'}\right) \e\left(-\frac{\overline{4bh'}\cdot ak^2}{q'}\right)
\e\left(h\cdot \frac{f^{\pm}(a,b,q,y_0+Y)}{q}\right) \right|
\end{equation}
and 
\begin{equation} \label{R2}
R_2(e,h,q,k)= \frac{h}{q} \cdot \sqrt{q'}\int\limits_{y_0}^{y_0+Y} \left| 
\sum\limits_{\substack{a,b\\ (a,b,q)\in S}} C_{a,b,q} \left(\frac{ab}{q'}\right) \e\left(-\frac{\overline{4bh'}\cdot ak^2}{q'}\right)
g_{h}^{\pm}(a,b,q,t) \right| \dif t. 
\end{equation}
To proceed further, we need to remove the weight functions $f^{\pm}$ and
$g_{h}^{\pm}$.

Recall \eqref{cubo}. We are now ready to impose a suitable constraint on the  partial derivatives
of $f^{\pm}$, wherever they are defined. We will assume that
\begin{equation} \label{fcond}
\left| \frac{\partial^{i+j+k}f^{\pm}}{\partial a^i \partial b^j
    \partial y^k} (a,b,q,y) \right| \leq \alpha^i \beta^j  
\tau^k F\quad 
\end{equation}
in $\mathcal{R}$ for $i,j,k\in \{0,1\}$ such that $i+j+k\not=0$, 
where $\alpha,\beta,\gamma,F$ are suitable non-negative numbers. We
shall also suppose that 
\begin{equation} \label{Hcond'}
H\ge \frac{q_0}{F}
\end{equation}
and set
\begin{equation} \label{Deltadef}
\Delta_H:=\left(1+\frac{HF\alpha A}{q_0}\right)\left(1+\frac{HF\beta B}{q_0}\right)\left(1+\frac{HF\tau Y}{q_0}\right).
\end{equation}
We now repeatedly apply partial summation with respect to $a$ and $b$ to remove the weight functions $f^{\pm}$ and $g_{h}^{\pm}$ in \eqref{R1} and \eqref{R2}. 
Then we interchange the integrals arising in this process with the sums on the right-hand side of \eqref{L1def}. 
Finally, we estimate the resulting integrals by multiplying their lengths with the supremums of their integrands,
which we bound using \eqref{fcond}. Taking \eqref{Hcond'} into consideration, we arrive at the bound for $\mathcal{F}_1^{\pm}$ in the following Theorem. 
By a parallel treatment, we obtain the corresponding bounds for
$\mathcal{F}_2^{\pm}$ and $\mathcal{F}_4^{\pm}$. 
  
\begin{thm} \label{generalbound} 
Assume the condition \eqref{fcond} and let $H$ satisfy \eqref{Hcond'}. Then we have
$$
\mathcal{F}^{\pm}\ll \mathcal{F}_1^{\pm}+\mathcal{F}_2^{\pm}+\mathcal{F}_4^{\pm},
$$
where 
$$
\mathcal{F}_i^{\pm} \ll \Delta_H \sup\limits_{(\eta,\theta)\in \mathbb{R}^2} \sum\limits_{h\le H} \sum\limits_{q} \sum\limits_{e\mid q} 
\frac{\xi_i(q')}{h\sqrt{q'}} \sum\limits_{k=0}^{q'-1} \min\left\{Y/e,\|k/q'\|^{-1}\right\} |R^{(i)}(\eta,\theta;e,h,q,k)|
$$
for $i=1,2,4$, with
\begin{align} \label{Ruvdhqk}
R^{(1)}(\eta,\theta;e,h,q,k)
&:= \sum\limits_{\substack{a\le \eta,\ b\le \theta\\ (a,b,q)\in S}} C_{a,b,q} \left(\frac{ab}{q'}\right) \e\left(-\frac{\overline{4bh'}\cdot ak^2}{q'}\right),\\
 \label{Ruvdhqk2}
R^{(2)}(\eta,\theta;e,h,q,k)
&:= \delta_k \sum\limits_{\substack{a\le \eta,\ b\le \theta\\ (a,b,q)\in S}} C_{a,b,q} \left(\frac{ab}{q'/2}\right) \e\left(-\frac{\overline{8bh'}\cdot ak^2}{q'/2}\right),\\
 \label{Ruvdhqk3}
R^{(4)}(\eta,\theta;e,h,q,k)
&:= (1-\delta_k) \sum\limits_{\substack{a\le \eta,\ b\le \theta\\ (a,b,q)\in S}} \epsilon_{h'ab}^{-1} C_{a,b,q} \left(\frac{q'}{ab}\right) 
\e\left(-\frac{\overline{bh'}\cdot ak^2}{4q'}\right).
\end{align}
\end{thm}

We are now in a position to deduce the bound in Theorem
\ref{t:fixedabqX}  for  fixed non-zero integers $a,b,q$ such that
$q\geq 1$ and $(ab,q)=1$.  In fact there is little extra effort
required to  handle a
more general quantity.   
Let  $J=(y_0,y_0+Y]$ be an interval with $Y\ge 1$  and 
assume that $f^{\pm} :  J\rightarrow 
\mathbb{R}$ are continuously differentiable functions with $f^{+}(y)
\ge f^{-}(y)$ for all $y\in J$.  
Set $I(y):=(f^{-}(y),f^+(y)]$ and $X(y):=f^+(y)-f^-(y)$. Assume that
$
| \frac{df^{\pm}}{dy}(y) | \le T$ for all $y\in J$. 
Then we have the following result.

\begin{cor*} 
Let $H>0$ and $
\Delta_H:=1+HTY/q.$
We have 
\begin{align*}
\sum\limits_{\substack{y\in J\\ (y,q)=1}}\ \sum\limits_{\substack{x\in I(y)\\ ax+by^2\equiv 0 \bmod{q}}} 1  =
~& 
\frac{1}{q} \sum\limits_{\substack{y\in J\\ (y,q)=1}} X(y)+O\left(\frac{Y}{H}\right)\\ 
&+ 
O\left(\Delta_H L(H)\sigma_{-1/2}(q)\left(\frac{Y}{\sqrt{q}}\cdot \tau(q)+\sqrt{q}L(q)\right)\right).
\end{align*}
where $L$ and $\sigma_{-1/2}$ are as in the statement of Theorem \ref{t:fixedabqX}. 
\end{cor*}

\begin{proof} Recall \eqref{mainterm} and \eqref{Eerror}. 
We set $f^{\pm}(a,b,q,y)=f^{\pm}(y)$, $q_0=q$, $F=q$, $\tau=T/F$ and $\alpha=\beta=0$ in 
the build-up to Theorem \ref{generalbound} . 
Estimating
$R^{(i)}(\eta,\theta;d,h,q,k)$ trivially by $O(1)$, and combining this with our work so far, we readily 
obtain the asymptotic estimate
\begin{align*} 
\frac{1}{q} \sum\limits_{\substack{y\in J\\ (y,q)=1}} X(y) +O\left(\frac{Y}{H}\right)+
O\left(\Delta_H\sum\limits_{h\le H} \frac{1}{h} \sum\limits_{e\mid q} \frac{e^{1/2}(he,q/e)^{1/2}}{q^{1/2}} 
\sum\limits_{k=0}^{q-1} \min\left\{\frac{Y}{e},\frac{q}{e(he,q/e)k}\right\}\right)
\end{align*}
for the double sum in the statement.
The second $O$-term here is seen to be
$$
\ll \Delta_H\cdot \frac{Y}{q^{1/2}} \sum\limits_{h\le H} \frac{1}{h} \sum\limits_{e\mid q} \frac{(he,q/e)^{1/2}}{e^{1/2}} + 
\Delta_H(\log H+1) (\log q+1) \sigma_{-1/2}(q) \sqrt{q},
$$
where the first term comes from the contribution of $k=0$ and the second one from the contribution of $k\not=0$. 
Since $(he,q/e)^{1/2}\le (h,q)^{1/2}e^{1/2}$, we have
$$
\sum\limits_{h\le H} \frac{1}{h} \sum\limits_{e\mid q} \frac{(he,q/e)^{1/2}}{e^{1/2}} \le \tau(q) \sum\limits_{h\le H}
\frac{(h,q)^{1/2}}{h}\ll \tau(q)\sigma_{-1/2}(q)\log (H+1). 
$$
This therefore completes the proof of the corollary.
\end{proof}

For Theorem \ref{t:fixedabqX} we take $J=(0,Y]$ and $I=(0,X]$, so that 
$f^{\pm}$ are constant  and we can set
$T=0$ and $\Delta_H=1$ in the corollary.  Taking $H=q$ we therefore obtain 
\begin{align*}
M_{1,2}(X,Y;a,b,q)
&= \frac{X}{q}\sum\limits_{\substack{y\in J\\ (y,q)=1}} 1+
O\left(L(q)\sigma_{-1/2}(q)\left(\frac{Y}{\sqrt{q}}\cdot \tau(q)+\sqrt{q}L(q) \right)\right).
\end{align*}
On noting that 
$$
\sum\limits_{\substack{y\in J\\ (y,q)=1}} 1 =\frac{\varphi(q)}{q} \cdot Y +O\left(\tau(q)\right),
$$
this completes the proof of Theorem \ref{t:fixedabqX} .

\section{Proof of  Theorem \ref{mainresult} }\label{proof2}

We now place ourselves in the setting of 
Theorem \ref{mainresult}, which is concerned with estimating $\cS$ in \eqref{Sdef} when 
$S$ is given by \eqref{eq:S} for 
fixed non-zero integers  $l,m,r,s,t$
for which $l,m,t\geq 1$
and $(rs,t)=1$.   Assume furthermore that \eqref{coeffcond} holds.
Now we can set
$$
a_0:=rU^l, \quad A:=(2^l-1)rU^l, \quad b_0:=sV^m, \quad B:=(2^m-1)sV^m, \quad q_0:=tW,\quad Q:=tW
$$
in \eqref{cubo}. With 
$
\tilde{f}^{\pm}$ as in \S\ref{s:intro}, we also set
$$ 
\tilde{I}(u,v,w,y):=I(r u^l,s v^m,t w,y), \quad 
\tilde{X}(u,v,w,y):=X(r u^l,s v^m,t w,y)
$$
and 
\begin{equation} \label{Dd}
D_{u,v}=d_{u,v}+|d_{u,v}|.
\end{equation}
Next we
observe that \eqref{fcond} is equivalent to
\eqref{fcondnew}
in $(U,2U]\times(V,2V]\times (W,2W]\times J$
for $i,j,k\in \{0,1\}$ such that $ i+j+k\not=0$, 
where
$$
\rho U=\frac{l}{2^l-1}\cdot \alpha A,  \quad  \sigma
V=\frac{m}{2^m-1}\cdot \beta B.
$$
In particular \eqref{Deltadef} has the same order of magnitude as 
\eqref{Deltadefnew} under this assumption, where we recall that  $l$ and $m$ are viewed as absolute constants. 

We may now write
$$
\mathcal{S}= \mathop{\sum\limits_{U<u\le 2U}\sum\limits_{V<v\le 2V}
\sum\limits_{W<w\le 2W}}_{(rs uv,tw)=1} d_{u,v}e_w 
\sum\limits_{\substack{y_0<y\le y_0+Y\\ (y,tw)=1}} \sum\limits_{\substack{x\in \tilde{I}(u,v,w,y)\\ ru^lx+sv^my^2\equiv 0 \bmod{t w}}} 1,
$$
and 
recall the decomposition in \eqref{decomp}. 
Using \eqref{mainterm}, the main term equals
\begin{equation} \label{mains}
\mathcal{M}=\mathop{\sum\limits_{U<u\le 2U}\sum\limits_{V<v\le 2V}\sum\limits_{W<w\le 2W}}_{(rs uv,tw)=1}  
\frac{d_{u,v}e_w}{tw}\sum\limits_{\substack{y_0<y\le y_0+Y\\ (y,tw)=1}} \tilde{X}(u,v,w,y).
\end{equation}
Using \eqref{Eerror} and \eqref{coeffcond}, the error term $\mathcal{E}$ is bounded by
\begin{equation} \label{firstes}
\mathcal{E}= \frac{Y}{H} \mathop{\sum\limits_{U<u\le 2U}\sum\limits_{V<v\le 2V}
\sum\limits_{W<w\le 2W}}_{(rs uv,tw)=1} |d_{u,v}e_w| \ll \frac{UVWY}{H}.
\end{equation}

We now turn to the error term $\mathcal{F}_1^{\pm}$. 
Using \eqref{coeffcond},  Theorem \ref{generalbound} and \eqref{Dd}, 
we see that 
$$
\mathcal{F}_1^{\pm}  
\ll
\Delta_H \sup\limits_{\substack{U\le \eta\le 2U\\ V\le \theta\le 2V}}\
\sum\limits_{h\le H} \ \sum\limits_{\substack{W<w\le 2W\\ (2rs,tw)=1}}
\ \sum\limits_{e\mid  tw}  
\frac{1}{h\sqrt{q'}}  \sum\limits_{k=0}^{q'-1}
\min\left\{Y/e,\|k/q'\|^{-1}\right\} |R(\eta,\theta;h',q',k)|.
$$
An application of \eqref{primes} therefore yields
\begin{equation} \label{afternotd}
\begin{split}
\mathcal{F}_1^{\pm}  
 \ll~& \frac{\Delta_H}{(tW)^{1/2}} \sup\limits_{\substack{U\le \eta\le
    2U\\ V\le \theta\le 2V}}\sum\limits_d \sum\limits_e  
\sum\limits_{\substack{h\le H\\ d\mid  he}} \frac{d^{1/2}e^{1/2}}{h}\\
&\times  
\sum\limits_{\substack{W<w\le 2W\\ (2rs,tw)=1\\ de\mid tw\\ (he,tw/e)=d}}  
\sum\limits_{k=0}^{q'-1} \min\left\{Y/e,\|k/q'\|^{-1}\right\} |R(\eta,\theta;h',q',k)|,
\end{split}
\end{equation}
where
\begin{equation} \label{qwnew}
d=(he,tw/e), \quad q'=\frac{tw}{de}, \quad h'=\frac{he}{d}
\end{equation}
and
$$
R(\eta,\theta;h',q',k)= 
\mathop{\sum\limits_{U<u\le \eta}\ \sum\limits_{V<v\le \theta}}_{(uv,tw)=1} 
D_{u,v} \left(\frac{u^lv^m}{q'}\right) \e\left(-\frac{\overline{4sv^mh'}\cdot ru^lk^2}{q'}\right).
$$
One derives similar bounds for $\mathcal{F}_2^{\pm}$ and $\mathcal{F}_4^{\pm}$ using \eqref{Ruvdhqk2} and \eqref{Ruvdhqk3} instead of \eqref{Ruvdhqk}. 
It will suffice to estimate $\mathcal{F}_1^{\pm}$ since the treatments of $\mathcal{F}_2^{\pm}$ and $\mathcal{F}_4^{\pm}$ will essentially be the same. We note 
that the right-hand side of \eqref{afternotd} is empty if $t$ is even, so we may assume that $t$ is odd.

In the next sections, we shall treat the contributions of $k=0$ and $k\not=0$ to the right-hand side of $\eqref{afternotd}$ separately. To
this end, we define
\begin{equation} \label{F0def}
\mathcal{K}_0:=\frac{\Delta_H Y}{(tW)^{1/2}} \sup\limits_{\substack{U\le \eta\le 2U\\ V\le \theta\le 2V}}\ \sum\limits_{d} \sum\limits_e 
\sum\limits_{\substack{h\le H\\ d\mid he}} \frac{d^{1/2}}{e^{1/2}h} \sum\limits_{\substack{W<w\le 2W\\ (2rs,w)=1\\ de\mid tw}}
\left|\mathop{\sum\limits_{U<u\le \eta}\ \sum\limits_{V<v\le \theta}}_{(uv,tw)=1} 
D_{u,v} \left(\frac{u^lv^m}{q'}\right) \right|
\end{equation}
and 
\begin{equation} \label{F1simpler}
\mathcal{K}_1:= \Delta_H (tW)^{1/2}\sup\limits_{\substack{U\le \eta\le 2U\\ V\le \theta\le 2V}} \sum\limits_{d} \sum\limits_{e} \sum\limits_{\substack{h\le H\\ d\mid eh}}  
\frac{1}{d^{1/2}e^{1/2}h} \sum\limits_{\substack{W<w\le 2W\\ (2rs,w)=1\\ de\mid tw\\ (he,tw/e)=d}} 
\sum\limits_{k=1}^{[q'/2]} \frac{1}{k} |R(\eta,\theta;h',q',k)|.
\end{equation}
Note that we have dropped the condition $ (he,tw/e)=d$ in
$\mathcal{K}_0$ but kept it in $\mathcal{K}_1$ since 
$R(\eta,\theta;h',q',k)$ is not well-defined if $(h',q')>1$. 

As a rule of thumb we expect $\mathcal{K}_0$ to dominate if $Y$ is large compared to $q_0$ and   $\mathcal{K}_1$ to dominate otherwise. Therefore, one would like to obtain
non-trivial bounds for $\mathcal{K}_0$ if $Y$ is large and non-trivial
bounds for $\mathcal{K}_1$ 
if $Y$ is small. Here we are mainly interested in the case of large $Y$.

\subsection{The contribution of $k=0$}

We aim to exploit cancellations coming from the Jacobi symbol. Our result will clearly 
depend on the parities
of the exponents $l$ and $m$.  We will establish the following bound.

\begin{prop} \label{F0final} 
We have 
$$
\mathcal{K}_0 \ll \frac{\Delta_H Y}{(tW)^{1/2}} \cdot (HtUVW)^{\varepsilon}\left(UVW^{1/2}+
U^{1-\{l/2\}}V^{1-\{m/2\}}W\right).
$$
\end{prop}

We will achieve this result by considering four different cases.
Suppose first that $l$ and  $m$ are odd. In this case, 
we shall treat the term $\mathcal{K}_0$ using Heath-Brown's large
sieve for real characters.  
First, we recall our assumption that $t$ is odd and note that $de$ is
also necessarily odd by our summation  
conditions  $(w,2)=1$ and $de\mid tw$. Now, using the oddness of the exponents $l$ and $m$,  
the multiplicativity of the Jacobi symbol and \eqref{qwnew}, we observe that
$$
\left(\frac{u^lv^m}{q'}\right)=\left(\frac{uv}{tde}\right)\left(\frac{uv}{w}\right) 
$$
since $(uv,tw)=1$. Furthermore we write
$$
\beta_{z}:=\left(\frac{z}{tde}\right)\sum\limits_{\substack{U<u\le
    \eta\\ V<v\le \theta\\ uv=z}} D_{u,v}. 
$$
Then it follows that
$$
\mathop{\sum\limits_{U<u\le \eta}\ \sum\limits_{V<v\le \theta}}_{(uv,tw)=1} D_{u,v} \left(\frac{u^lv^m}{q'}\right) = 
\sum\limits_{UV<z\le 4UV} \beta_{z}\left(\frac{z}{w}\right),
$$
where we note that the coprimality condition $(uv,tw)=1$ is implied by
the Jacobi symbols. We further note that
$\beta_{z}=O( z^{\varepsilon})$
by \eqref{coeffcond} and \eqref{Dd}.
Next we write
$$
\left|\sum\limits_{UV<z\le 4UV} \beta_{z}\left(\frac{z}{w}\right)\right|=
\alpha_{w}  \sum\limits_{UV<z\le 4UV} \beta_{z}\left(\frac{z}{w}\right),
$$
where $\alpha_{w}$ is a suitable complex number with
$|\alpha_{w}|=1$.
The inner triple sum in
\eqref{F0def} now takes the form 
$$
\sum\limits_{\substack{W<w\le 2W\\ (2rs,w)=1\\ de\mid tw}}
\left|\mathop{\sum\limits_{U<u\le \eta}\ \sum\limits_{V<v\le \theta}}_{(uv,tw)=1}
D_{u,v} \left(\frac{u^lv^m}{q'}\right) \right|= \sum\limits_{\substack{W<w\le 2W\\ (2rs,w)=1\\ de\mid tw}}\alpha_{w}  \sum\limits_{UV<z\le 4UV} \beta_{z}\left(\frac{z}{w}\right).
$$
We observe that
$tw \equiv 0 \bmod{de} $ if and only if $w \equiv 0 \bmod{de/(de,t)}$.
Hence 
$$
\sum\limits_{\substack{W<w\le 2W\\ (2rs,w)=1\\ de\mid tw}} 
\alpha_{w}  \sum\limits_{UV<z\le 4UV} \beta_{z}\left(\frac{z}{w}\right)=
\sum\limits_{\substack{W/j<w\le 2W/j\\ (2rs,jw)=1}} 
\tilde{\alpha}_{w}  \sum\limits_{UV<z\le 4UV} \tilde{\beta}_{z}\left(\frac{z}{w}\right),
$$
where 
$$
j=\frac{de}{(de,t)}, \quad \tilde{\alpha}_{w}=\alpha_{jw}, \quad \tilde{\beta}_{z}=\beta_{z}
\cdot \left(\frac{z}{j}\right).
$$

Recalling that $\beta_{z}=O( z^{\varepsilon})$
and applying  Lemma \ref{lsreal}, we deduce that 
$$
\sum\limits_{\substack{W/j<w\le 2W/j\\ (2rs,jw)=1}} 
\tilde{\alpha}_{w}  \sum\limits_{UV<z\le 4UV} \tilde{\beta}_{z}\left(\frac{z}{w}\right)\ll (UVW)^{\varepsilon}
\left(\frac{UVW^{1/2}}{j^{1/2}}+\frac{U^{1/2}V^{1/2}W}{j}\right).
$$
Combining our work in \eqref{F0def},  
and noting that $de\mid tw$, we obtain the preliminary bound
$$
\mathcal{K}_0 \ll  \frac{\Delta_H Y}{(tW)^{1/2}} \cdot (UVWH)^{\varepsilon}\left(UVW^{1/2}+U^{1/2}V^{1/2}W\right) 
\sum\limits_{\substack{d,e\\ de\le 2tW}} 
\sum\limits_{\substack{h\le H\\ d\mid he}} \frac{d^{1/2}}{e^{1/2}hj^{1/2}}.
$$
But
\begin{align*} 
\sum\limits_{\substack{d,e\\ de\le 2tW}}  \sum\limits_{\substack{h\le
    H\\ d\mid he}} \frac{d^{1/2}}{e^{1/2}hj^{1/2}} 
=
\sum\limits_{\substack{d,e\\ de\le 2tW}}  \sum\limits_{\substack{h\le H\\ d\mid he}} \frac{(de,t)^{1/2}}{eh} 
&\ll (HtW)^{\varepsilon}
\sum\limits_{e\le 2tW}  \sum\limits_{h\le H} \frac{(he^2,t)^{1/2}}{eh} \\ 
&\le (HtW)^{\varepsilon}\sum\limits_{e\le 2tW}  \frac{(e,t)}{e}
    \sum\limits_{h\le H} \frac{(h,t)^{1/2}}{h} \\ 
&\ll (HtW)^{2\varepsilon}.
\end{align*}
This therefore gives
\begin{equation} \label{finalF0bound}
\mathcal{K}_0 \ll \frac{\Delta_H Y}{(tW)^{1/2}} \cdot (HtUVW)^{\varepsilon}\left(UVW^{1/2}+U^{1/2}V^{1/2}W\right),
\end{equation}
which is satisfactory for Proposition \ref{F0final}.

Next suppose that $m$ is odd and $l$ is even. Then we have 
$$
\left(\frac{u^lv^m}{q'}\right) = \chi_0(u)\left(\frac{v}{q'}\right),
$$
where $\chi_0$ is the principal character modulo $q'$. Hence, it is not possible to exploit the summation over $u$.
Therefore, we sum over $u$ trivially and estimate the term 
$$
\sum\limits_{\substack{\substack{W<w\le 2W\\ (2rs,w)=1\\ de\mid tw}}} 
\left|\sum\limits_{\substack{V<v\le \theta\\ (v,tw)=1}} 
D_{u,v} \left(\frac{v}{q'}\right) \right|
$$
using Lemma  \ref{lsreal}, just as above. In this way we arrive at the same bound for
$\mathcal{K}_0$, where the term $U^{1/2}$ in \eqref{finalF0bound} is replaced by $U$, as required.
If $l$ is odd and $m$ is even then the situation is the same, with the roles
of $u$ and $v$ being interchanged. Thus, in this case, the term
$V^{1/2}$ in \eqref{finalF0bound} needs to be replaced 
by $V$.

Finally suppose that $l$ and $m$ are both even.
$$
\left(\frac{u^lv^m}{q'}\right) = \chi_0(uv),
$$
where $\chi_0$ is the principal character modulo $q'$. Hence, in this case we have no cancellations at all in $\mathcal{K}_0$, and the
only possibility is to estimate trivially. Here the term
$UVW^{1/2}+U^{1/2}V^{1/2}W$ in \eqref{finalF0bound} 
needs to be replaced by $UVW$.

This completes the proof of Proposition \ref{F0final} . 
We note from \eqref{Ruvdhqk3} that when dealing with the contribution corresponding to $\mathcal{K}_0$ in $\mathcal{F}_4^{\pm}$, the roles of $ab$ and $q$ in the Jacobi symbol are flipped. 
The oddness condition on $m=ab$ in Lemma \ref{lsreal} will be satisfied since $(ab,q)$=1 and $4\mid q$, whence $(ab,2)=1$ in this case.

\subsection{The contribution of $k\not=0$}


We  first estimate the contribution $\mathcal{K}_1$ of $k\not=0$
trivially, by bounding all coefficients $D_{u,v}$ and $e_w$ and the
characters occurring in $R(\eta,\theta;h',q',k)$ by $O(1)$.  
Rearranging summations and dropping several summation conditions, we obtain
$$
\mathcal{K}_1\ll \Delta_H (tW)^{1/2}UV \sum\limits_{h\le H} \frac{1}{h} 
\sum\limits_{W<w\le 2W} \sum\limits_{k\le tw} \frac{1}{k} \sum\limits_{\substack{d,e\\ de\mid tw}} \frac{1}{d^{1/2}e^{1/2}},
$$
which therefore implies the following bound.

\begin{prop} \label{F1final} 
We have 
$
\mathcal{K}_1\ll \Delta_H (tW)^{1/2}UVW(HtW)^{\varepsilon}.
$
\end{prop}

A non-trivial saving can be obtained if $UV$ is large compared to
 $q_0$ and $d_{u,v}$ factorises in the form 
 \eqref{Duvfact}, which we now assume.
By \eqref{Dd} we have
$$
R(\eta,\theta;h',q',k)=R_1(\eta,\theta;h',q',k)+R_2(\eta,\theta;h',q',k),
$$
where 
\begin{align*}
R_1(\eta,\theta;h',q',k)
&:=\mathop{\sum\limits_{U<u\le \eta}\ \sum\limits_{V<v\le \theta}}_{(uv,tw)=1} 
d_u'\tilde{d}_v \left(\frac{u^lv^m}{q'}\right) \e\left(-\frac{\overline{4sv^mh'}\cdot ru^lk^2}{q'}\right),\\
R_2(\eta,\theta;h',q',k)&:=\mathop{\sum\limits_{U<u\le \eta}\ \sum\limits_{V<v\le \theta}}_{(uv,tw)=1} 
|d_u'| \cdot |\tilde{d}_v| \left(\frac{u^lv^m}{q'}\right) \e\left(-\frac{\overline{4sv^mh'}\cdot ru^lk^2}{q'}\right).
\end{align*}
We focus here on bounding $R_1$, the estimation of $R_2$ being similar.
 
We begin by  writing
$$
\e\left(-\frac{\overline{4sv^mh'}\cdot ru^lk^2}{q'}\right) 
= \e\left(-\frac{\overline{4sv^mh'}\cdot ru^lk'}{q''}\right),
$$
where 
\begin{equation} \label{copri}
k':=\frac{k^2}{(q',k^2)}, \quad  q''=\frac{q'}{(q',k^2)}. 
\end{equation}
Now we write the additive character in terms of multiplicative characters via
\begin{align*}
\e\left(-\frac{\overline{4sv^mh}\cdot ru^lk'}{q''}\right)= & \frac{1}{\varphi(q'')} \sum\limits_{\chi\bmod{q''}} \overline{\chi}(-\overline{4sv^mh}\cdot ru^lk')\tau(\chi)\\ 
= & \frac{1}{\varphi(q'')} \sum\limits_{\chi\bmod{q''}} \chi(-4sh\overline{rk'})  \overline{\chi}^l(u)\chi^m(v)\tau(\chi).
\end{align*}
It follows that
$$
R_1(\eta,\theta; h',q',k)=  \frac{1}{\varphi(q'')} \sum\limits_{\chi\bmod{q''}} \chi(-4sh\overline{rk'}) \tau(\chi)
\sum\limits_{\substack{U<u\le \eta\\ (u,tw)=1}} d_u'' \overline{\chi}^l(u) \sum\limits_{\substack{V<v\le \theta\\ (v,tw)=1}} 
\tilde{\tilde{d}}_v  \chi^m(v),
$$
where
$
d_u'':=d_u'(\frac{u}{q'})^l$ and $ \tilde{\tilde{d}}_v:=\tilde{d}_{v}(\frac{v}{q'})^m.
$
Note that for every fixed $n\in \mathbb{N}$ and every character $\chi_1$ mod $q''$ there are at most $O\left({q''}^{\varepsilon}\right)$ characters $\chi$ mod $q''$ with 
$\chi_1=\chi^n$. Therefore, using Cauchy--Schwarz and the well-known
bounds 
$|\tau(\chi)|\le \sqrt{q''}$ and $\varphi(q'')\gg q''^{1-\varepsilon}$, we deduce that
$$
|R_1(\eta,\theta;h',q',k)|\ll {q''}^{-1/2+\varepsilon}
\left(
\hspace{-0.1cm}
\sum\limits_{\chi\bmod{q''}} \left|
    \sum\limits_{\substack{U<u\le \eta\\ (u,tw)=1}} d_u''
    \overline{\chi}(u)\right|^2\right)^{1/2} 
\hspace{-0.2cm}
\left(
\hspace{-0.1cm}
\sum\limits_{\chi\bmod{q''}} \left|\sum\limits_{\substack{V<v\le \theta\\ (v,tw)=1}} 
\tilde{\tilde{d}}_v  \chi(v)\right|^2\right)^{1/2}.
$$
Now using the large sieve for fixed modulus (see Iwaniec and Kowalski
\cite[page 179]{IwKo}, for example), together with $|d_u''|$,
$|\tilde{\tilde{d}}_{v}|\le 1$,
we deduce that 
\begin{align*}
R_1(\eta,\theta;h',q',k)
\ll & {q''}^{-1/2+\varepsilon}(q'+U)^{1/2}(q'+V)^{1/2}(UV)^{1/2}. 
\end{align*}
The same estimate holds for $R_2(\eta,\theta;h',q',k)$ on redefining $d_u''$ and $\tilde{\tilde{d}}_v$ accordingly.
Hence, using \eqref{qwnew} and \eqref{copri}, it follows that
\begin{align*}
\sum\limits_{k=1}^{[q'/2]} \frac{1}{k} |R(\eta,\theta;h',q',k)| & \ll q'^{-1/2+\varepsilon}(q'+U)^{1/2}(q'+V)^{1/2}(UV)^{1/2} \sum\limits_{k=1}^{[q'/2]} \frac{(q',k^2)^{1/2}}{k}\\
& \ll d^{1/2}e^{1/2} (tW)^{-1/2+2\varepsilon}(tW+U)^{1/2}(tW+V)^{1/2}(UV)^{1/2},
\end{align*}
where we have estimated the $k$-sum by $O\left({q'}^{\varepsilon}\right)$.
Plugging the last line  into \eqref{F1simpler}, rearranging the summations and dropping several summation conditions, we obtain
$$
\mathcal{K}_1\ll \Delta_H  (tW)^{\varepsilon}(tW+U)^{1/2}(tW+V)^{1/2}(UV)^{1/2} \sum\limits_{h\le H} \frac{1}{h} 
\sum\limits_{W<w\le 2W} \sum\limits_{\substack{d,e\\ de\mid tw}} 1.
$$
This yields the following result, 
which improves Proposition \ref{F1final} if $UV$ is larger than $q_0=tW$.

\begin{prop} \label{anotherF1bound} We have 
$\mathcal{K}_1\ll \Delta_H (tW+U)^{1/2}(tW+V)^{1/2}(UV)^{1/2}W(HtW)^{\varepsilon},$
if \eqref{Duvfact} holds. 
\end{prop}

\subsection{Conclusion}\label{s:final}

Now we are ready to prove our final asymptotic estimate for $\mathcal{S}$. First, combining Propositions \ref{F0final}, \ref{F1final} and \ref{anotherF1bound},
we get
$$
\mathcal{F}_1^{\pm}\ll 
\mathcal{K}_0+\mathcal{K}_1\ll \mathcal{T},
$$
where  
$\mathcal{T}$ is as in the statement of Theorem \ref{mainresult}.
The same bound holds for $\mathcal{F}_2^{\pm}$ and
$\mathcal{F}_4^{\pm}$. Hence, using Theorem \ref{generalbound}, 
we obtain 
$\mathcal{F}^{\pm}\ll \mathcal{T}.$
Combining this with \eqref{decomp},  \eqref{mains} and \eqref{firstes}, we arrive at the statement of 
Theorem \ref{mainresult}.


We end this section by discussing 
conditions under which we may expect the main term to dominate the error term in Theorem 
\ref{mainresult}. In many applications, the length $\tilde{X}(u,v,w,y)$ of the $x$-interval will be of size 
$
\tilde{X}(u,v,w,y)\asymp X\le q_0=tW,
$ 
for some fixed $X>0$, and the parameters in \eqref{fcondnew} will satisfy
\begin{equation} \label{albeta}
F\asymp X, \quad \rho\asymp U^{-1}, \quad \sigma\asymp V^{-1}, \quad \tau \asymp Y^{-1}.
\end{equation}
Moreover, in generic applications $U$ and $V$ will be shorter than the
modulus, and so we further suppose that $U\leq tW$ and $V\leq tW$. 

If there is not much cancellation in the sums over the coefficients,
then the expected size of the main term in \eqref{mains} is 
$$
\mathcal{M} \asymp \frac{UVWXY}{q_0}.
$$
For the first $O$-term on the right-hand side of the asymptotic
formula in Theorem \ref{mainresult} to be dominated by this  
we need $H$ just slightly larger than $q_0/X$. The choice 
$$
H=\frac{q_0^{1+\varepsilon}}{X}
$$
would be satisfactory. 
Then
$\Delta_H\ll q_0^{\varepsilon}$,  by
\eqref{Deltadefnew} and \eqref{albeta}.
Now, for $\mathcal{T}$ to be smaller than $\mathcal{M}$, we need
$$
q_0^{1+\varepsilon}\le \min\left\{U^{2\{l/2\}}V^{2\{m/2\}}X^2,Z\right\} \quad \mbox{and} \quad q_0^{\varepsilon}t^{1/2} \le X,
$$
where 
$$
Z:=\begin{cases}  (UV)^{1/4}(XY)^{1/2}, & \mbox{if \eqref{Duvfact} holds 
and $UV\geqq tW$,}\\ 
(XY)^{2/3},
 & \mbox{in general.} \end{cases}
$$

\end{document}